\documentclass[11pt]{amsart}
\usepackage[english]{babel}
\usepackage[T1]{fontenc}    
\usepackage[utf8]{inputenc}  
\usepackage{fancyhdr}       
\usepackage{amsfonts}       
\usepackage{amssymb}       	
\usepackage{amsmath}        
\usepackage{amsthm}         
\usepackage{mathtools}    
\usepackage[all,cmtip]{xy}
\usepackage{mathrsfs}
\usepackage{graphicx}
\usepackage{xcolor}
\usepackage[shortlabels]{enumitem} 

\setenumerate[0]{label=(\alph*)}

\usepackage{bbm} 

\calclayout
\usepackage{hyperref}       
\hypersetup{
    colorlinks=true,
    linkcolor=blue,
    filecolor=blue,      
    urlcolor=blue,
}

\usepackage[multiple]{footmisc} 

\usepackage{todonotes}

\usepackage{cleveref}

\usepackage{tikz-cd} 
\usetikzlibrary{decorations.pathmorphing} 

\newtheorem{theorem}[subsubsection]{Theorem} 

\newtheorem{maintheorem}{Theorem}

\theoremstyle{definition}

\newtheorem{example}[subsubsection]{Example} 

\newtheorem{lemma}[subsubsection]{Lemma}
\newtheorem{corollary}[subsubsection]{Corollary}

\numberwithin{equation}{subsection}

\theoremstyle{remark}
\newtheorem{remark}[subsubsection]{Remark} 

\numberwithin{equation}{section}

\DeclareMathOperator{\Spec}{Spec}

\DeclareMathOperator{\tr}{Tr}
\DeclareMathOperator{\cent}{Cent}

\DeclareMathOperator{\sym}{Sym}

\DeclareMathOperator{\red}{\text{red}} 

\newcommand{\fkp}{\mathfrak{p}}

\newcommand{\gzip}{\text{$G$-$\mathtt{Zip}$}} 

\newcommand{\GL}{\mathrm{GL}} 

\newcommand{\ha}{\mathrm{Ha}} 

\usepackage{hyperref}
\hypersetup{
    colorlinks=true,
    linkcolor=blue,
    filecolor=blue,      
    urlcolor=blue,
}
\usepackage{cleveref}
 
\providecommand{\keywords}[1]
{
  \small	
  \textbf{\textit{Keywords---}} {#1}
}

\newcommand{\GZip}{\mathop{\text{$G$-{\tt Zip}}}\nolimits}

\renewcommand{\AA}{\mathbf{A}}

\newcommand{\FF}{\mathbf{F}}
\newcommand{\GG}{\mathbf{G}}

\newcommand{\NN}{\mathbf{N}}

\newcommand{\QQ}{\mathbf{Q}}
\newcommand{\RR}{\mathbf{R}}

\newcommand{\XX}{\mathbf{X}}

\newcommand{\ZZ}{\mathbf{Z}}

\newcommand{\Lscr}{{\mathscr L}}

\newcommand{\Vscr}{{\mathscr V}}

\newcommand{\Fcal}{{\mathcal F}}

\newcommand{\Hcal}{{\mathcal H}}
\newcommand{\Ical}{{\mathcal I}}
\newcommand{\Jcal}{{\mathcal J}}
\newcommand{\Kcal}{{\mathcal K}}

\newcommand{\Ocal}{{\mathcal O}}

\newcommand{\Ycal}{{\mathcal Y}}

\newcommand{\leftexp}[2]{{\vphantom{#2}}^{#1}{#2}}

\newcommand{\iw}{\leftexp{I}{W}}

\usepackage{csquotes}       
\usepackage[style=trad-alpha]{biblatex} 
\title{Systems of Hecke eigenvalues on subschemes of Shimura varieties}
\author{Stefan Reppen}
\address[S. Reppen]{Graduate School of Mathematical Sciences, the University of Tokyo}
\email{stefan.reppen@gmail.com}

\makeatletter
\let\c@equation=\c@subsubsection

\let\c@figure=\c@subsubsection

\begin{document}

\begin{abstract}
We show that the systems of Hecke eigenvalues that appear in the coherent cohomology with coefficients in automorphic line bundles of any mod $p$ abelian type compact Shimura variety at hyperspecial level are the same as those appearing in any Hecke-equivariant closed subscheme. We also prove analogous results for noncompact Shimura varieties or nonclosed subschemes, such as Ekedahl-Oort strata, length strata and central leaves.
\end{abstract}
\maketitle

\keywords{Shimura varieties, Hasse invariants, Hecke actions, Ekedahl-Oort stratification, central leaves}


\section{Introduction}
The mod $p$ fiber of an abelian type Shimura variety at good reduction comes equipped with an action of the prime-to-$p$ Hecke algebra. 
Several well-studied subschemes - such as Ekedahl-Oort strata, length strata, Newton strata and central leaves - are invariant under this action. 
By considering coherent cohomology for automorphic bundles, we thus obtain several Hecke modules, and it is interesting to relate the eigenvalues that appear in these modules to those appearing in the cohomology of the Shimura variety. 
In a famous letter from Serre to Tate \cite{serre-letter.to.Tate}, Serre shows that the systems of Hecke eigenvalues that appear in the cohomology of a mod $p$ modular curve are the same as those appearing in the cohomology of the supersingular locus. This result was generalized by Goldring-Koskivirta in \cite[Theorem 11.1.1]{goldring.koskivirta-strata.hasse}. Letting $S_K$ denote the special fiber of the Kisin-Vasiu integral model of a Hodge-type Shimura variety at good reduction, the authors of ibid. show that the systems of Hecke eigenvalues appearing in the cohomology of $S_K$ are the same as those appearing in the cohomology of the unique zero-dimensional Ekedahl-Oort stratum.
\footnote{They also showed that these eigenvalues are the same as those appearing in the cohomology of the toroidal compactification of $S_K$. For the purpose of brevity, in the introduction we omit any discussion on compactifications.} As pointed out by Terakado-Yu in \cite{terakado.yu-hecke.eigensystems}, the method of ibid. can be applied to any Hecke-equivariant smooth subscheme of dimension 0 equipped with a nowhere vanishing Hecke-invariant section of the Hodge line bundle. They use this observation to prove that the systems of Hecke eigenvalues appearing in the cohomology of the Shimura variety are the same as those appearing in the cohomology of any central leaf contained in the basic Newton stratum \cite[Theorem 5.4]{terakado.yu-hecke.eigensystems}.

In the aforementioned works, the subschemes considered are smooth of dimension 0. 
In this text we prove new, analogous results for higher dimensional and possibly singular subschemes. To state the results, let $T$ be a maximal torus of the special fiber of the reductive group defining the Shimura variety. Let $L$ be the Levi subgroup associated to the special fiber of the Hodge cocharacter. Let $X_{+,L}^*(T)$ denote the $L$-dominant characters and for each $\eta\in X_{+,L}^*(T)$ let $\Vscr(\eta)$ denote the associated automorphic line bundle. 
\begin{maintheorem}\label{theorem: main theorem intro}(\Cref{theorem: main theorem.in.text})
Let $S_K$ be the special fiber of the integral canonical model of a compact abelian type\footnote{See \Cref{section: shimura datum}.} Shimura variety, at hyperspecial level $K$. For any (prime-to-$p$) Hecke-equivariant closed subscheme $Y_K\subset S_K$, the systems of Hecke eigenvalues that appear in 
\[
\bigoplus_{\eta\in X_{+,L}^*(T)}H^0(S_K,\Vscr(\eta)) \,\,\, \text{ and } \,\,\, \bigoplus_{\eta\in X_{+,L}^*(T)}H^0(Y_K,\Vscr(\eta))
\]
are the same.
\end{maintheorem}
\Cref{theorem: main theorem intro} applies in particular to the closure of all subschemes mentioned in the first paragraph of the introduction, including higher-dimensional central leaves. We also obtain results regarding nonclosed subschemes and noncompact Shimura varieties (see \Cref{theorem: general cases}). We single out the following corollary.
\begin{maintheorem}\label{theorem: second theorem intro}(\Cref{corollary: eo length strata and central leaves})
Let $S_K$ be the special fiber of the integral canonical model of an abelian type Shimura variety, at hyperspecial level $K$. Let $Y_K$ be either an Ekedahl-Oort stratum or a length stratum, and let $C_K$ be a central leaf. If $S_K$ is of Hodge type or is proper, then the systems of Hecke eigenvalues that appear in 
\[
\bigoplus_{\eta\in X_{+,L}^*(T)}H^0(S_K,\Vscr(\eta)) \,\,\, \text{ and } \,\,\, \bigoplus_{\eta\in X_{+,L}^*(T)}H^0(Y_K,\Vscr(\eta))
\]
are the same. If $S_K$ is proper and $C_K$ is closed in the Ekedahl-Oort stratum containing it, then these are the same as the systems of Hecke eigenvalues appearing in
\[
\bigoplus_{\eta\in X_{+,L}^*(T)}H^0(C_K,\Vscr(\eta)).
\]
%
\end{maintheorem}
The proofs of \cite[Theorem 11.1.1]{goldring.koskivirta-strata.hasse} and \cite[Theorem 5.4]{terakado.yu-hecke.eigensystems} proceed in two steps; a going-down argument relating eigenvalues on the Shimura variety to eigenvalues on the subschemes in question (i.e., the zero-dimensional Ekedahl-Oort stratum \cite{goldring.koskivirta-strata.hasse} or the zero-dimensional central leaves \cite{terakado.yu-hecke.eigensystems}), and a going-up argument in the other direction. A key idea in the going-up argument is to use that there is a nowhere vanishing Hasse invariant on the subscheme in question. For the going-down argument, the proofs relate Hecke eigenforms on the Shimura variety to Hecke eigenforms of the sheaf of differentials via the conormal exact sequence, building on arguments of Ghitza \cite{ghitza-hecke.eigenvalues}. Since the subschemes are smooth of dimension 0, existence of Hasse invariants is not needed in the going-down arguments. 

In this text we consider higher dimensional and possibly singular subschemes $Y_K$. One 
thus faces two problems; (1) there is no nowhere vanishing Hecke-invariant section on $Y_K$, and (2) the conormal sequence is neither exact nor consists of locally free sheaves. The key idea to surpass these problems is to perform a further going-down argument. By intersecting $Y_K$ with the smallest length stratum that it intersects nontrivially, and by successively removing the \textit{nonsingular} points of this intersection, we obtain a smooth, closed subscheme $Z_K\subset Y_K$ which admits a nowhere vanishing Hecke-equivariant section. We thereafter pass from the ideal sheaf defining $Z_K$ in $Y_K$ to the ideal sheaf defining $Z_K$ in $S_K$ and modify the going-down arguments of \cite{goldring.koskivirta-strata.hasse} to relate Hecke eigenvalues on $Y_K$ respectively $S_K$ to eigenvalues on $Z_K$. This step crucially uses the existence of Hasse invariants, in contrast to the going-down argument of ibid. Finally, the going-up argument of ibid. applies verbatim to relate eigenvalues on $Z_K$ to $Y_K$ and $S_K$ respectively.
\subsection{Outline}
\Cref{section; preliminaries} introduces notation and background, and contains proofs of some preliminary results regarding Hecke-equivariance. \Cref{section: main section} contains the proof of \Cref{theorem: main theorem intro} in \Cref{subsection: proof of thm A} and results for nonclosed subschemes and noncompact Shimura varieties in \Cref{section: proof of thm B}.
\section{Preliminaries}\label{section; preliminaries}
\subsection{Notation related to the Shimura datum}\label{section: shimura datum}
Let $(\mathbf{G},\mathbf{X})$ be a Shimura datum of abelian type and assume that $\GG=\GG^c$.
\footnote{The group $\GG^c$ is the cuspidal quotient of $\GG$ by its maximal central torus that is split over $\RR$ but anisotropic over $\QQ$. The assumption that $\GG=\GG^c$ is a technical assumption that provides a smooth surjective morphism to the stack of $G$-zips, and we will not use it elsewhere in the text. Hence, we refer to \cite[Section 3.3]{imai.kato.youcis-prismatic.realization} for the definition of $\GG^c$. For example, if $(\GG,\XX)$ is of Hodge type, then we always have that $\GG=\GG^c$. One could also drop this assumption and in the main theorems only consider automorphic bundles arising from characters of $\GG^c$.} 
For the rest of the paper, fix a prime $p\neq 2$ over which $\GG$ is unramified and let $K_p\subset \mathbf{G}(\QQ_p)$ be a hyperspecial subgroup. For each sufficiently small open compact subgroup $K^p\subset \mathbf{G}(\AA_f^p)$ let $K=K_pK^p$ and let $S_K$ denote the special fiber of the Kisin-Vasiu integral canonical model of the Shimura variety associated to $(\mathbf{G},\mathbf{X}, K)$ (\cite{kisin-hodge-type-shimura},\cite{vasiu}). The construction of the integral model is compatible with changing the prime-to-$p$ level, and for two hyperspecial subgroups $K'\subset K$ the transition map $\pi_{K'/K}\colon S_{K'}\to S_K$ is finite \'etale. For each $g\in \mathbf{G}(\AA_f^p)$ we have a natural isomorphism $g\colon S_{gKg^{-1}}\to S_K$. 
\subsection{Reminder on the theory of \texorpdfstring{$G$}{G}-zips} We give a minimal account on this theory and refer the reader to the original papers \cite{pink.wedhorn.ziegler.zip.data} and \cite{pink.wedhorn.ziegler.additional} for more details. Let $G$ be the special fiber of a reductive model of $\GG$ and let $\mu$ be the special fiber of (a representative of the conjugacy class of) an integral model of the Hodge cocharacter. Let $\kappa$ be the field of definition of $\mu$. Let $P,P^+$ be the pair of opposite parabolics determined by $\mu$, and let $L=\cent_G(\mu)$ be their Levi quotient. 
For our purposes, it suffices to say that the \textit{stack of $G$-zips of type $\mu$} $\GZip^\mu$ is a smooth algebraic $\kappa$-stack of dimension $0$ isomorphic to a quotient stack $[E\backslash G]$, where $E$ is a certain subgroup of $P\times (P^+)^{(p)}$, which acts on $G$ by conjugation. Thus, by restriction, any representation $r\colon P\to \GL(V)$ gives a representation of $E$, and thus a vector bundle $\Vscr(r)$ on $\GZip^\mu$.

If $(\GG,\XX)$ is of Hodge type, then by \cite{zhang} we have for each $K$ a smooth surjective morphism
\[
\zeta_K\colon S_K\to \GZip^\mu.
\]
The author thanks Alex Youcis for pointing out that the existence of such a $\zeta_K$ for the abelian type case follows from \cite{gardner.madapusi.mathew-moduli.of.truncated},\cite{imai.kato.youcis-prismatic.realization}, which is the sole reason for our assumption that $\GG=\GG^c$ (see also \cite[Theorem 3.4.7]{shen.zhang} for a weaker statement). 
These morphisms are compatible with the transition morphisms of the system $(S_K)_K$. 

An \textit{automorphic vector bundle} on $S_K$ is a vector bundle that is isomorphic to $\zeta_K^*\Vscr(r)$ for some representation $r\colon P\to \GL(V)$. For example, the sheaf of differentials $\Omega_{S_K}$ is the automorphic vector bundle associated to the $L$-representation $(\text{Lie}(G)/\text{Lie}(P))^{\vee}$. At one point in \Cref{lemma: good filtraation} we shall need \Cref{lemma: good filtration L-module} below regarding this $L$-module. To state it, 
fix a maximal torus $T$ and let $X_{+,L}^{*}(T)$ denote the $L$-dominant characters of $T$. 
\begin{lemma}\label{lemma: good filtration L-module}(\cite[Lemma 11.2.3]{goldring.koskivirta-strata.hasse})
    For each $\eta\in X_{+,L}^*(T)$ let $V_{\eta}$ denote the highest weight module corresponding to $\eta$. For all $n\geq 0$, the $L$-module \[
    \sym^n\Big((\text{Lie}(G)/\text{Lie}(P))^{\vee}\Big)\otimes V_{\eta}\] admits an $L$-stable filtration
    \[
    \sym^n\Big((\text{Lie}(G)/\text{Lie}(P))^{\vee}\Big)\otimes V_\eta=V_0\supset V_1\supset ... \supset V_m \supset 0
    \]
    such that $V_i/V_{i+1}\cong V_{\eta_i}$ for some $\eta_i\in X_{+,L}^*(T)$.
\end{lemma}
\textit{The Hodge line bundle} $\omega\coloneqq \det\Omega_{S_K}$ is an automorphic vector bundle. By abuse of notation, we denote by $\omega$ also the line bundle on $\GZip^\mu$ whose pullback along $\zeta_K$ is $\det\Omega_{S_K}$.
\subsubsection{Ekedahl-Oort and length stratifications}
Fix a Borel subgroup $B$ with $T\subset B$, whose $\kappa$-fiber is contained in $P$. Let $W$ denote the Weyl group of $(G,T)$ and let $\Delta$ denote the simple roots corresponding to $B$. 
Let $I\subset \Delta$ be the subset corresponding to $P$ and let ${}^IW$ denote the set of minimal length representatives of $W_I\backslash W$, where $W_I\subset W$ is the subgroup generated by $I$. There exists a partial order $\preceq$ on $\iw$ 
which induces a topology on $\iw$. Let $z$ denote the element of maximal length in $\iw$. For each $w\in \iw$, let $G_w$ be the $E$-orbit of of $wz^{-1}$. The map $w\mapsto G_w$ induces a homeomorphism of $\iw$ with the underlying topological space of $\GZip^\mu$. The \textit{Ekedahl-Oort stratification} is the corresponding stratification
\[
\gzip^\mu = \bigcup_{w\in \iw} [E\backslash G_w].
\]
The closure relation is given by
\begin{equation*}
[E\backslash \overline{G}_w] = \bigcup_{w'\preceq w} [E\backslash G].
\end{equation*}
Let $l\colon W\to \NN$ denote the length function. For each $j\in \NN$, let 
\[
G_j\coloneqq \bigcup_{\{w\in \iw:l(w)=j\}}G_w.
\]
The \textit{length $j$ stratum} of $\GZip^\mu$ is the substack $[E\backslash G_j]$. One obtains thus the \textit{length stratification} of $\GZip^\mu$ and the closure of the length $j$ stratum is given by
\begin{equation}\label{equation: closure length stratum}
[E\backslash \overline{G}_j] = \bigcup_{i\leq j} [E\backslash G_i].
\end{equation}
Since the morphisms $\zeta_K\colon S_K\to \GZip^\mu$ are smooth and surjective, we obtain via pullback the Ekedahl-Oort and length stratifications on $S_K$.
\subsubsection{Hasse invariants}
We shall make repeated use of the following results of \cite{goldring.koskivirta-strata.hasse}.
\begin{lemma}\label{lemma: eo Hasse invariant}(\cite[Theorem 3.2.3]{goldring.koskivirta-strata.hasse})
    There exists an integer $r\geq 1$ and for each $w\in \iw$ a section $\ha_w$ of $H^0([E\backslash \overline{G}_w],\omega^r)$ whose nonvanishing locus is exactly $[E\backslash G_w]$.
    
\end{lemma}
\begin{lemma}\label{lemma: length Hasse invariant}(\cite[Proposition 5.2.1]{goldring.koskivirta-strata.hasse})
    There exists an integer $r\geq 1$ and for each $j$ a section $\ha_j$ of $H^0([E\backslash \overline{G}_j],\omega^r)$ whose nonvanishing locus is exactly $[E\backslash G_j]$.
\end{lemma}
The sections in the previous lemmas are called the \textit{(Ekedahl-Oort) strata Hasse invariants} respectively the \textit{length Hasse invariants}. We obtain via pullback along $\zeta_K$ the corresponding Hasse invariants on $S_K$.
\subsection{Reminder on central leaves}
Since we will not work explicitly with the central leaves, we simply sketch their definition in the Hodge type case. For details and definitions in the case of abelian type, see \cite[Section 4]{shen.zhang}.

Let $A_K\to S_K$ denote the universal abelian variety over $S_K$. For any $x\in S_K(\overline{\kappa})$, the \textit{central leaf} of $x$, denoted $C(x)$, is the set of points $y\in S_K(\overline{\kappa})$ such that there is an isomorphism of Dieudonn\'e modules $\mathbb{D}(A_{K,x}[p^\infty])(W(\overline{\kappa}))\cong \mathbb{D}(A_{K,y}[p^\infty])(W(\overline{\kappa}))$. Each central leaf is locally closed in $S_{K,\overline{\kappa}}$. Since the Ekedahl-Oort strata parametrizes the $p$-torsion $A_K[p]$ and Newton strata parametrizes isogeny classes of the $p$-divisible groups $A_{K}[p^\infty]$, we see that any central leaf is contained in a unique Ekedahl-Oort stratum and a unique Newton stratum. The dimension of $C(x)$ has an explicit combinatorial description in terms of the Newton map, which we omit here. We point out that central leaves in the unique Newton stratum of minimal dimension are smooth and zero-dimensional.
\subsection{\texorpdfstring{$\GG(\AA_f^p)$}{G}-equivariance and the Hecke action}\label{section: equivariance}
Let $(\Box_K)_K$ be a system of sheaves on $(S_K)_K$ or a system of subschemes of $(S_K)_K$. 
We say that $(\Box_K)_K$ is \textit{$\GG(\AA_f^p)$-equivariant} if $\pi_{K'/K}^*\Box_K=\Box_{K'}$ and $g^*\Box_K=\Box_{gKg^{-1}}$ for all $K'\subset K$ and all $g\in \GG(\AA_f^p)$. Equivalently, $\Box_K$ is $\GG(\AA_f^p)$-equivariant if its pullback to $\varprojlim_{K^p}S_{K_pK^p}$ along the projection morphism is $\GG(\AA_f^p)$-equivariant. 

Similarly a $\GG(\AA_f^p)$-equivariant morphism between such $\GG(\AA_f^p)$-equivariant systems is said to be $\GG(\AA_f^p)$-equivariant if it is compatible with $\pi_{K'/K}$ and $g$ for all $K'\subset K$ and all $g\in \GG(\AA_f^p)$.

We also refer to $\GG(\AA_f^p)$-equivariant objects as \textit{Hecke-stable}, or \textit{Hecke-equivariant}.
\begin{example}\label{example: equiv systems}
    \begin{enumerate}
        \item\label{item: eo strata} Ekedahl-Oort strata and length strata are $\GG(\AA_f^p)$-equivariant. More generally, if $\Box$ is any substack of $\GZip^\mu$ or a sheaf on it, we obtain a $\GG(\AA_f^p)$-equivariant system $(\Box_K)_K$ via pullback, $\Box_K\coloneqq \zeta_K^*\Box$. In particular, automorphic vector bundles are $\GG(\AA_f^p)$-equivariant.
        \item\label{item: newton strata} Newton strata are $\GG(\AA_f^p)$-equivariant.
        \item\label{item: central leaves} The central leaf of any point $x\in S_K(\overline{\FF}_p)$ is $\GG(\AA_f^p)$-equivariant. 
    \end{enumerate}
\end{example}
\subsubsection{Definition of the Hecke action}
Let $\text{Ram}(\GG)$ denote the set of places at which $\GG$ ramifies. For each $\nu \in \text{Ram}(\GG)\cup\{p\}$ and each hyperspecial subgroup $\Kcal_\nu\subset \GG(\QQ_\nu)$, let $\Hcal_\nu=\Hcal_\nu(\GG_\nu,\Kcal_\nu,\ZZ_p)$ be the unramified Hecke algebra of $\GG$ at $\nu$, with $\ZZ_p$ coefficients. Let $\Hcal$ denote the unramified, prime-to-$p$ global Hecke algebra, defined as the restricted tensor product
\[
\Hcal\coloneqq \bigotimes_{\nu \in \text{Ram}(\GG)\cup\{p\}} \Hcal_\nu.
\]

For $K=K_pK^p$ as in \Cref{section: shimura datum}, and $g\in \GG(\QQ_p)$, let $K'=K\cap  gKg^{-1}$ and consider the diagram
\begin{equation*}
    \begin{tikzcd}
        & S_{K'} \arrow[ld, "\pi_{K'/K}"'] \arrow[rd, "\pi_{K'/gKg^{-1}}"] & & \\
        S_K & & S_{gKg^{-1}} \arrow[r, "g"] & S_K
    \end{tikzcd}
\end{equation*}
Suppose that $(\Fcal_K)_K$ is a Hecke-equivariant system of coherent sheaves on $(S_K)_K$. 
To each $g\in \GG(\AA_f^p)$ we associate the Hecke operator $T_g$, which acts on $H^0(S_K,\Fcal_K)$ by push-pull
\[
T_g \coloneqq \tr(g\circ \pi_{K'/gKg^{-1}})\circ \pi_{K'/K}^*.
\]
Since $\Hcal_\nu$ is generated by characteristic functions of double cosets represented by $g\in G(\QQ_\nu)$ we obtain an action of $\Hcal$ on $H^0(S_K,\Fcal)$. This is called the \textit{Hecke-action}.

Given a field $\kappa$ and an $\Hcal$-module $M$ which is a finite dimensional vector space over $\kappa$, we say that a system of Hecke eigenvalues $(b_T)_T$ appears in $M$ if there is some finite extension $\kappa'/\kappa$ and $m\in M\otimes \kappa'$ such that $T\cdot m=b_T m$ for all $T\in \Hcal$.
\begin{remark}
If $\varphi\colon M\to M'$ is a surjective (resp. injective) morphism of $\Hcal$-modules, then every system of Hecke eigenvalues that appears in $M'$ (resp. $M$) also appears in $M$ (resp. $M'$) (cf. \cite[Lemma 8.3.1]{goldring.koskivirta-strata.hasse}).
\end{remark}
\subsubsection{Some results regarding \texorpdfstring{$\GG(\AA_f^p)$}{G}-equivariance}
%
%
%
%
%
%
For a reduced scheme $X$ over $\kappa$, let $X^{\text{sm}}$ denote the maximal smooth open subscheme of $X$. 
\begin{lemma}\label{lemma: equivariant stuff}
Let $(Y_K)_K$ be a system of $\GG(\AA_f^p)$-equivariant subschemes. We have the following:
\begin{enumerate}
    \item\label{item: closures} The closure $\overline{Y}_K$ is $\GG(\AA_f^p)$-equivariant.
    \item\label{item: intersection} If $Y_K$ is closed and $Z_K$ is another $\GG(\AA_f^p)$-equivariant closed subscheme, then $Z_K\cap Y_K$ is $\GG(\AA_f^p)$-equivariant.
    \item\label{item: reduced subscheme} If $Y_K$ is closed, the reduced subscheme $(Y_K)_{\red}$ is $\GG(\AA_f^p)$-equivariant.
    \item\label{item: smooth locus} If $Y_K$ is closed and reduced, then $Y_K^{\text{sm}}$ is $\GG(\AA_f^p)$-equivariant.
    \item\label{item: complement} If $Y_K$ is closed and $U_K\subset Y_K$ is an open $\GG(\AA_f^p)$-equivariant subscheme, then $Z_K=Y_K\setminus U_K$ with its reduced induced scheme structure is $\GG(\AA_f^p)$-equivariant.
\end{enumerate}
\end{lemma}
\begin{proof}
For \ref{item: closures}, let $Z_K\coloneqq \overline{Y}_K$ and let $\pi \coloneqq \pi_{K'/K}\colon S_{K'}\to S_K$ be a transition map. Since $\pi$ is finite and \'etale, $\pi(Z_{K'})=Z_K$ as closed subsets of $S_K$. Hence, if $X$ is an arbitrary reduced scheme and $f\colon X\to S_{K'}$ is a morphism with image contained in $Z_{K'}$, then $\pi\circ f$ has image in $Z_K$. By the universal property of $Z_K$ with the reduced subscheme structure on $\overline{Y}_K$, we obtain a commutative diagram
\[
\begin{tikzcd}
    X \arrow[rrd, bend left = 30, "f"] \arrow[ddr, bend right = 30] & & \\
     & Z_K \times_{S_K} S_{K'} \arrow[r] \arrow[d] & S_{K'} \arrow[d] \\
     & Z_K \arrow[r] & S_K.
\end{tikzcd}
\]
By the universal property of $Z_K\times_{S_K} S_{K'}$ as a fiber product, we see that it satisfies the universal property of $Z_{K'}$. This shows \ref{item: closures}.

The other assertions follow similary thanks to the transition morphisms being finite and \'etale:

For \ref{item: intersection}, let $\Ical_K$ and $\Jcal_K$ denote the ideal sheaves of $Z_K$ respectively $Y_K$ in $S_K$. A system of closed subschemes is $\GG(\AA_f^p)$-equivariant if and only if the corresponding system of ideal sheaves is. Since $Z_K\cap Y_K$ is the subscheme corresponding to $\Ical_K + \Jcal_K$, \ref{item: intersection} follows. 

For \ref{item: reduced subscheme}, since 
the transition morphisms $\pi_{K'/K}\colon S_{K'}\to S_K$  are \'etale, we find that $(Y_K)_{\red} \times_{S_K} S_{K'}$ is reduced. Hence, $(Y_K)_{\red} \times_{S_K} S_{K'}= (Y_{K'})_{\red}$.

For \ref{item: smooth locus}, let $\pi \coloneqq \pi_{K'/K}$. Since $\pi$ is surjective, maximality of $Y_{K'}^{\text{sm}}$ implies that 
$\pi(Y_{K'}^{\text{sm}})\supset Y_K^{\text{sm}}$.  Since $\pi$ is \'etale, 
each point in $\pi(Y_{K'}^{\text{sm}})$ is smooth. Hence, by maximality of $Y_K^{\text{sm}}$, we have that $\pi(Y_{K'}^{\text{sm}})=Y_K^{\text{sm}}$ and thus maximality of $Y_{K'}^{\text{sm}}$ gives that $\pi^*Y_K^{\text{sm}}=Y_{K'}^{\text{sm}}$ as topological spaces. Since $(Y_K)_K$ is equivariant, $\pi^*\Ocal_{Y_K^{\text{sm}}}=\Ocal_{Y_{K'}^{\text{sm}}}$ and we see that $\pi^*Y_K^{\text{sm}} = Y_{K'}^{\text{sm}}$ as schemes.

Finally, we show \ref{item: complement}. Since the transition maps are finite and \'etale, they are open and closed. Hence, $\pi^*Z_K$ identifies with the closed subset $Z_{K'}$. Since the transition maps are \'etale, $\pi^*Z_K$ is reduced. By uniqueness of the reduced subscheme structure, we thus see that $\pi^*Z_K= Z_{K'}$. 
\end{proof}
\begin{lemma}\label{lemma: equivariant.morphisms}
Suppose that $(Z_K)_K$ is a $\GG(\AA_f^p)$ equivariant system of closed subschemes of $(S_K)_K$ corresponding to ideal sheaves $(\Ical_K)_K$. Then, for all integers $n\geq 0$ the following sheaves and morphisms thereof are $\GG(\AA_f^p)$-equivariant:
\begin{enumerate}
    \item the power $\Ical_K^n$,
    \item the short exact sequence
    \[
    0\to \Ical_K^{n+1}\to \Ical_K^n \to \Ical_K^n/\Ical_K^{n+1}\to 0,
    \]
    \item the natural surjection 
    \[
    \sym^n(\Ical_K/\Ical_K^2)\to \Ical_K^n/\Ical_K^{n+1}\to 0,
    \]
    and
    \item the conormal exact sequence
    \[
    \Ical_K/\Ical_K^2\to \Omega_{S_K}|_{Z_K} \to \Omega_{Z_K}\to 0
    \]
\end{enumerate}
Furthermore, if $(Y_K)_K$ is another equivariant system of closed subschemes of $(S_K)_K$ corresponding to ideal sheaves $(\Jcal_K)_K$ such that $\Jcal_K \subset \Ical_K$, then the natural quotient map induces a $\GG(\AA_f^p)$-equivariant map
\[
\Ical_K^n/\Ical_K^{n+1}\to \Jcal_K^n/\Jcal_K^{n+1}\to 0.
\]
\end{lemma}
\begin{proof}
The transition morphisms $\pi_{K'/K}\colon Z_{K'}\to Z_K$ are \'etale, hence $\pi_{K'/K}^*\Omega_{Z_K}=\Omega_{Z_K'}$. Since $\sym^n$ commutes with arbitrary pullbacks, the statements thus follow from flatness of the transition morphisms.
\end{proof}
\subsection{Miscellaneous}
For the convenience of the reader, 
we collect here a two elementary statements that are used in the proof of \Cref{theorem: main theorem intro} and \Cref{theorem: second theorem intro}.

Recall that given a scheme $X$ and a line bundle $\Lscr$ on $X$, a global section $s$ of $\Lscr$ is called \textit{injective} (also often called regular) if the map $\Ocal_X \to \Lscr$, $f\mapsto fs$ is injective.
\begin{lemma}\label{lemma: nowhere vanishing. injective}
If $X$ is a reduced scheme and $s$ is a nowhere vanishing global section of a line bundle $\Lscr$ on $X$, then $s$ is injective.
\end{lemma} 
\begin{proof}
The question is local on $X$ so we can assume that $X=\Spec A$ for a reduced ring $A$ and that $s\in A$. Since $s$ is nowhere vanishing, for all prime ideals $\fkp\subset A$, $s\notin \fkp$. Hence, if $fs=0$ for some $f\in A$, then $f$ is contained in all prime ideals of $A$, hence $f$ is nilpotent. Since $A$ is reduced, $f=0$.
\end{proof}
\begin{lemma}\label{lemma: injectivity of sym}
Suppose that $X$ is a regular scheme and that $\iota\colon \Vscr\hookrightarrow \Vscr'$ is an inclusion of locally free sheaves. Then, for all $n\geq 0$, $\sym^n(\iota)$ is injective.
\end{lemma}
\begin{proof}
Let $\Kcal \coloneqq \ker\sym^n(\iota)$. Since $\sym^n(\Vscr)$ is locally free, it suffices to show that $\Kcal$ is torsion, i.e. that  $\Kcal_x$ is torsion for all $x\in X$. Fix a point $x$ in $X$. Since $X$ is regular the local ring $\Ocal_{X,x}$ is a domain. Let $Q$ denote the fraction field. Since $\sym^n$ takes injections of vector spaces to injections and commutes with localizations, and since localization is exact and $\sym^n$ commutes with taking stalks, we see that 
\begin{equation}
\begin{aligned}  
0 &= \ker\Big(\sym^n(\Vscr_x\otimes Q) \to \sym^n(\Vscr_x' \otimes Q)\Big) \\
&= \ker\Big(\sym^n(\Vscr_x)\otimes Q \to \sym^n(\Vscr_x') \otimes Q\Big) \\
&= \ker\Big(\sym^n(\Vscr_x) \to \sym^n(\Vscr_x') \Big) \otimes Q \\
&=\Kcal_x\otimes Q.
\end{aligned}
\end{equation}
Hence, $\Kcal_x$ is torsion, and thus $\Kcal=0$.
\end{proof}
\section{Systems of Hecke eigenvalues: Proof of the main theorems}\label{section: main section}
\subsection{Notation}
Let $(Y_K)_K \subset (S_K)_K$ be a $\GG(\AA_f^p)$-equivariant system of subschemes. 
We denote by $\omega = \det\Omega_{S_K}$ the Hodge line bundle and we use the same notation when we restrict $\omega$ to $Y_K$, as opposed to the line bundle $\det \Omega_{Y_K}$ which need not be automorphic. Let 
\[
M(Y_K) \coloneqq \bigoplus_{\eta\in X_{+,L}^*(T)}H^0(Y_K, \Vscr(\eta)).
\] 
Suppose that $(Z_K)_K$ is another system of equivariant subschemes. If any system of Hecke eigenvalues that appear in $M(Z_K)$ also appears in $M(Y_K)$, then we write $M(Z_K)\rightsquigarrow_{\Hcal} M(Y_K)$. If also $M(Y_K)\rightsquigarrow_{\Hcal} M(Z_K)$, then we write $M(Z_K) \leftrightsquigarrow_{\Hcal} M(Y_K)$.
\subsection{Proof of \Cref{theorem: main theorem intro}}\label{subsection: proof of thm A}
\begin{lemma}\label{lemma: good filtraation}
If $(b_T)_T$ is a system Hecke eigenvalues appearing in $H^0(Y_K, \sym^n(\Omega_{S_K}|_{Y_K})\otimes \Vscr(\eta))$ for some $\eta\in X_{+,L}^*(T)$, then it appears in $H^0(Y_K, \Vscr(\eta_i))$ for some $\eta_i\in X_{+,L}^*(T)$.
\end{lemma}
\begin{proof}
Let $V_\bullet$ be the filtration on $\sym^n\Big((\text{Lie}(G)/\text{Lie}(P))^{\vee}\Big) \otimes V_\eta$ from \Cref{lemma: good filtration L-module}. The projections $V_0\twoheadrightarrow V_i$ induce morphisms $\pi_i\colon \sym^n(\Omega_{S_K}|_{Y_K})\otimes \Vscr(\eta)\to \Vscr(\eta_i)$ which are $\GG(\AA_f^p)$-equivariant since they are defined over $\GZip^\mu$. Hence, we obtain a $\GG(\AA_f^p)$-equivariant map
\[
H^0(Y_K, \sym^n(\Omega_{S_K}|_{Y_K})\otimes \Vscr(\eta)) \to \bigoplus_{i=0}^m H^0(Y_K, \Vscr(\eta_i)),
\]
which is readily seen to be injective.
\end{proof}
\begin{lemma}\label{lemma: going-down}
Suppose that $Z_K\subset Y_K$ are two $\GG(\AA_f^p)$-equivariant subschemes such that $Z_K$ is proper and smooth. 
If there exists an $r>0$ and an injective Hecke-equivariant section $h$ of $H^0(Z_K,\omega^r)$, 
then $M(Y_K)\rightsquigarrow_{\Hcal} M(Z_K)$. 
\end{lemma}
\begin{proof}
Let $f\in H^0(Y_K,\Vscr(\eta))$ be a Hecke eigenform with system of Hecke eigenvalues $(b_T)_T$. Let $\Ical$ denote the ideal sheaf of $Z_K$ in $S_K$, and let $\Jcal$ denote the ideal sheaf of $Z_K$ in $Y_K$. Since $Y_K$ is Noetherian, there is a largest integer $n\geq 0$ such that $\overline{f}$ is nonzero in $H^0(Z_K,\Jcal^n/\Jcal^{n+1}\otimes \Vscr(\eta))$. By \Cref{lemma: equivariant.morphisms} we have a $\GG(\AA_f^p)$-equivariant surjection 
\[
\sym^n(\Ical/\Ical^2)\twoheadrightarrow \Ical^n/\Ical^{n+1}\twoheadrightarrow \Jcal^n/\Jcal^{n+1}.
\]
Since $Z_K$ is assumed proper, Serre vanishing shows that there is some $a>0$ such that 
\[
H^1(Z_K,\Jcal^n/\Jcal^{n+1}\otimes \Vscr(\eta)\otimes \omega^{ar})=0.
\]
We thus have a surjection
\begin{equation}\label{equation: surjection sym to i}
H^0(Z_K,\sym^n(\Ical/\Ical^2)\otimes \Vscr(\eta)\otimes \omega^{ar}) \twoheadrightarrow H^0(Z_K,\Jcal^n/\Jcal^{n+1}\otimes \Vscr(\eta)\otimes \omega^{ar})
\end{equation}
of $\Hcal$-modules. By assumption on $h$ the system $(b_T)_T$ of Hecke eigenvalues appears in $H^0(Z_K,\Jcal^n/\Jcal^{n+1}\otimes \Vscr(\eta)\otimes \omega^{ar})$ and thus by \ref{equation: surjection sym to i} it also appears in $H^0(Z_K,\sym^n(\Ical/\Ical^2)\otimes \Vscr(\eta)\otimes \omega^{ar})$. Consider the morphism
\[
\psi\colon \sym^n(\Ical/\Ical^2)\otimes \Vscr(\eta) \otimes \omega^{ar} \to \sym^n(\Omega_{S_K}|_{Z_K})\otimes \Vscr(\eta) \otimes \omega^{ar}
\]
obtained from the conormal exact sequence. Since $Z$ is smooth, $\Ical/\Ical^2$ is locally free and thus so is $\sym^n(\Ical^n/\Ical^{n+1})$. By \Cref{lemma: injectivity of sym} $\psi$ induces an injection of $\Hcal$-modules
\[
H^0(Z_K, \sym^n(\Ical/\Ical^2)\otimes \Vscr(\eta) \otimes \omega^{ar}) \to H^0(Z_K, \sym^n(\Omega_{S_K}|_{Z_K})\otimes \Vscr(\eta) \otimes \omega^{ar}).
\]
Hence, the system of Hecke eigenvalues $(b_T)_T$ also appears in $H^0(Z_K, \sym^n(\Omega_{S_K}|_{Z_K})\otimes \Vscr(\eta) \otimes \omega^{ar})$. By \Cref{lemma: good filtraation}, this implies that $(b_T)_T$ appears in $M(Z_K)$.
\end{proof}
\begin{lemma}\label{lemma: going-up}
Suppose that $Z_K\subset Y_K$ are two $\GG(\AA_f^p)$-equivariant subschemes such that $Y_K$ is proper and $Z_K\subset Y_K$ is closed. If there exists an $r>0$ and an injective Hecke equivariant section $h$ of $H^0(Z_K,\omega^r)$, then $M(Z_K)\rightsquigarrow_{\Hcal} M(Y_K)$. 
\end{lemma}
\begin{proof}
The arguments of \cite[Lemma 11.2.6]{goldring.koskivirta-strata.hasse} applies verbatim, as follows. Let $f\in H^0(Z_K,\Vscr(\eta))$ be a Hecke eigenform. By the assumption on $h$, for all $a\geq 0$, $fh^a\in H^0(Z_K,\Vscr(\eta)\otimes \omega^{ra})$ is a nonzero Hecke eigenform with the same Hecke eigenvalues as $f$. Let $\Ical$ be the ideal sheaf of $Z_K$ in $Y_K$, and let $\iota\colon Z_K\hookrightarrow Y_K$ denote the immersion. Consider the $\GG(\AA_f^p)$-equivariant short exact sequence
\[
0\to \Ical \otimes \Vscr(\eta) \to \Vscr(\eta) \to \iota_{*}\Ocal_Z \otimes \Vscr(\eta) \to 0.
\]
By Serre vanishing, $H^1(Y_K, \Ical \otimes \Vscr(\eta) \otimes \omega^{ra})=0$ for $a$ large enough. Hence, we have a surjection of $\Hcal$-modules
\[
H^0(Y_K, \Vscr(\eta) \otimes \omega^{ra})\twoheadrightarrow H^0(Z_K, \Vscr(\eta) \otimes \omega^{ra})
\]
for $a$ large enough. Thus, the system $(b_T)_T$ appears in $M(Y_K)$.
\end{proof}
\begin{remark}
In \Cref{lemma: going-down} and \Cref{lemma: going-up}, one can replace $\omega$ with any ample, Hecke-equivariant line bundle.
\end{remark}
\begin{theorem}\label{theorem: main theorem.in.text}(\Cref{theorem: main theorem intro})
Suppose that $(\GG,\XX)$ is of compact-type. Let $(Y_K)_K \subset (S_K)_K$ be an arbitrary $\GG(\AA_f^p)$-equivariant system of closed subschemes. We have that $M(Y_K)\leftrightsquigarrow_{\Hcal} M(S_K)$.
\end{theorem}
\begin{proof}
Let 
\[
S_K = S_{K,\dim(S_K)} \cup ... \cup S_{K,0}
\]
be the length stratification of $S_K$, and let 
\[
i_0\coloneqq \min\{ i : Y_K\cap \overline{S}_{K,i}\neq \emptyset \}.
\]
Define inductively the subschemes 
\[
Z_K^{(0)} \coloneqq (Y_K\cap \overline{S}_{K,i_0})_{\red},
\]
and for all $i\geq 1$,
\[
Z_K^{(i)}\coloneqq (Z_K^{(i-1)}\setminus Z_K^{(i-1),\text{sm}})_{\red}.
\]
By \Cref{lemma: equivariant stuff}, $Z_K^{(i)}$ form an equivariant system of closed subschemes and we get a decreasing chain of closed subschemes
\[
Y_K \supset Z_K^{(0)} \supset Z_K^{(1)} \supset \cdots 
\]
Since $Y_K$ is Noetherian, there is a minimal integer $m$ such that $Z_K^{(m)}=Z_K^{(m')}$ for all $m'\geq m$. Let $Z_K\coloneqq Z_K^{(m)}$. If $Z_K$ is not smooth, then $Z^{(m),\text{sm}}\subsetneq Z^{(m)}$ which implies that $Z^{(m+1)}\subsetneq Z^{(m)}$, contradicting the minimality of $m$. Hence, $Z_K$ is smooth.

By \Cref{lemma: length Hasse invariant} there exists an integer $r>0$ and a section $\ha_{i_0}$ of $H^0(\overline{S}_{K,i_0},\omega^r)$ whose nonvanishing locus is exactly $S_{K,i_0}$. Let $h$ denote the restriction of $\ha_{i_0}$ to $Z_K$. We claim that $h$ is nowhere vanishing. If $i_0=0$, then $\ha_{i_0}$ is nowhere vanishing, and thus so is $h$. Hence, we may assume that $i_0>0$. By \Cref{equation: closure length stratum} we have that
\[
\overline{S}_{K,i_0}=\bigcup_{i\leq i_0} S_i =S_{K,i_0}\cup \bigcup_{i\leq i_0-1} S_i = S_{K,i_0} \cup \overline{S}_{K,i_0-1}
\]
and thus $Z(\ha_{i_0}) = \overline{S}_{K,i_0-1}$ 
(as topological spaces). Hence, if there is some point $x$ in $Z_K$ on which $h$ vanishes, then 
$x$ lies in $Y_K \cap \overline{S}_{K,i_0-1}$. By definition of $i_0$, $Y_K \cap \overline{S}_{K,i_0-1}=\emptyset$. 
Hence, $h\in H^0(Z_K, \omega^r)$ is a nowhere vanishing section and
%
thus, by \Cref{lemma: nowhere vanishing. injective}, $h$ is injective.  Since $S_K$ is proper, so is any closed subscheme. Hence, we apply \Cref{lemma: going-down} and \Cref{lemma: going-up} on the pair $(Z_K)_K\subset (Y_K)_K$ respectively $(Z_K)_K \subset (S_K)_K$ to obtain that
\[
M(Y_K) \leftrightsquigarrow_{\Hcal} M(Z_K) \leftrightsquigarrow_{\Hcal} M(S_K).
\]
\end{proof}
\begin{corollary}
For any two $\GG(\AA_f^p)$-equivariant systems of closed subschemes $(Z_K)_K,(Y_K)_K\subset (S_K)_K$ we have that $M(Y_K)\leftrightsquigarrow_{\Hcal} M(Z_K)$.
\end{corollary}
\begin{proof}
By \Cref{theorem: main theorem.in.text} applied to the two systems respectively we obtain that
\[
M(Y_K)\leftrightsquigarrow_{\Hcal} M(S_K) \leftrightsquigarrow_{\Hcal} M(Z_K).
\]
\end{proof}
\subsection{Results for nonclosed subschemes and noncompact Shimura varieties}\label{section: proof of thm B}
As far as we know, the theory of toroidal compactifications of abelian type Shimura varieties has not yet been developed. Hence, in this section we assume that if $(\GG,\XX)$ is not of compact-type, then it is of Hodge type.
\subsubsection{Compactifications}
Let $S_K^\Sigma$ denote a smooth toroidal compactification of $S_K$ \cite[Theorem 1]{pera-toroidal.compactifications}. Let $D=D_K^\Sigma$ denote the boundary divisor of $S_K^\Sigma$ relative $S_K$. 
By \cite[Theorem 6.2.1]{goldring.koskivirta-strata.hasse} the morphism $\zeta_K$ extends to a morphism
\[
\zeta_K^\Sigma\colon S_K^\Sigma \to \GZip^\mu
\]
which is smooth by \cite[Theorem 1.2]{andreatta-two.mod.p.period.maps}.

The Hecke-action extends to $S_K^\Sigma$ \cite[Section 8.1.5]{goldring.koskivirta-strata.hasse}, so we extend all notations from the previous sections to $S_K^\Sigma$. To ensure that $M(S_K)\leftrightsquigarrow_{\Hcal} M(S_K^\Sigma)$, we make the same assumption as in \cite{goldring.koskivirta-strata.hasse},\cite{terakado.yu-hecke.eigensystems}; if $(\GG,\XX)$ is neither of compact-type nor of PEL-type, then assume that there exists a $\GG(\AA_f^p)$-equivariant divisor $D'$ such that $D_{\red}'=D$ and such that $\omega^k(-D')$ is ample on $S_K^\Sigma$ for all $k$ large enough.
\begin{theorem}\label{theorem: general cases}
Let $Y_K\subset S_K$ be an arbitrary locally closed $\GG(\AA_f^p)$-equivariant subscheme. The following holds.
\begin{enumerate}
    \item\label{item: going.down.generally} We have that $M(S_K)\rightsquigarrow_{\Hcal} M(Y_K)$.
    \item\label{item: going.up.boundary.assumption} If furthermore $Y_K\subset S_K^\Sigma$ is closed, 
    then $M(Y_K)\rightsquigarrow_{\Hcal} M(S_K)$.
    \item\label{item: basic locus} If $Y_K$ meets the basic locus, then $M(S_K)\leftrightsquigarrow_{\Hcal} M(Y_K)$.
    \item\label{item: pullback from gzip} If $Y_K=\zeta_K^*\Ycal_K$ for a closed substack $\Ycal$ of $\GZip^\mu$, then $M(Y_K)\leftrightsquigarrow_{\Hcal} M(S_K)$.
    \item\label{item: nonvanishing locus} Suppose that $M(Y_K)\leftrightsquigarrow_{\Hcal} M(S_K)$ and that there exists a Hecke-equivariant line bundle $\Lscr$ and a Hecke-equivariant section $h\in H^0(Y_K,\Lscr)$. Then $M(Y_K)\leftrightsquigarrow_{\Hcal} M(Y_{K,h})$, where $Y_{K,h}$ denotes the nonvanishing locus of $h$.
\end{enumerate}
\end{theorem}
\begin{proof}
Let $Y_K^\Sigma$ denote the closure of $Y_K$ in $S_K^\Sigma$. The arguments of \Cref{theorem: main theorem.in.text} applies verbatim to show that $M(S_K^\Sigma)\rightsquigarrow_{\Hcal} M(Y_K^\Sigma)$. By \cite[Theorem 11.1.1]{goldring.koskivirta-strata.hasse}, $M(S_K)\leftrightsquigarrow_{\Hcal} M(S_K^\Sigma)$, hence $M(S_K)\rightsquigarrow_{\Hcal} M(Y_K^\Sigma)$. Since $Y_K$ is open and dense in $Y_K^\Sigma$, $M(Y_K^\Sigma)\rightsquigarrow_{\Hcal} M(Y_K)$. 
Since $Y$ does not meet the boundary, 
the ample divisor $\omega(-D')$ identifies with $\omega$ when restricted to $Y$. Thus, if $Y$ is closed, then the arguments of \Cref{theorem: main theorem.in.text} applies. This shows \ref{item: going.down.generally} and \ref{item: going.up.boundary.assumption}.

For \ref{item: basic locus}, let $x\in Y_K$ be a point contained in the basic locus. Since $Y_K$ is Hecke-equivariant the central leaf $C(x)$ of $x$ is contained in $Y_K$. Since $x$ is contained in the basic locus, $C(x)$ is smooth, closed and zero dimensional. Hence, \ref{item: going.down.generally} and \ref{item: going.up.boundary.assumption} (or \cite[Theorem 5.4]{terakado.yu-hecke.eigensystems}) implies that $M(S_K)\leftrightsquigarrow_{\Hcal} M(C(x))$. 
By \Cref{lemma: going-down} respectively \Cref{lemma: going-up} we see that $M(Y_K)\rightsquigarrow_{\Hcal} M(C(x))\rightsquigarrow_{\Hcal} M(\overline{Y}_K)$. Since $Y$ is open and dense in $\overline{Y}$, $M(\overline{Y}_K)\rightsquigarrow_{\Hcal} M(Y_K)$ and the statement follows.

To prove \ref{item: pullback from gzip}, let $\Ycal \subset \GZip^{\mu}$ be a closed substack. Then there is some $w\in \iw$ such that the Ekedahl-Oort stratum $[E\backslash G_w]$ is contained in $\Ycal$. Since $[E\backslash \overline{G}_w]=\bigcup_{w'\preceq w}[E\backslash G_{w'}]$ and $\Ycal$ is closed, $\Ycal$ contains $[E\backslash G_e]$. Hence, $Y_K$ contains the smallest Ekedahl-Oort stratum, and the result follows from \ref{item: basic locus}.

Finally, we prove \ref{item: nonvanishing locus}. Since $Y_{K,h}$ is open in $Y$, the restrction map $H^0(Y_K,\Vscr(\eta))\to H^0(Y_{K,h}, \Vscr(\eta))$ is injective. This shows that $M(Y_K)\rightsquigarrow_{\Hcal} M(Y_{K,h})$. Conversely, suppose that $f\in H^0(Y_{K,h},\Vscr(\eta))$ is a Hecke eigenform. By \cite[Theorem 9.3.1]{grothendieck-egaI} there is some $n\geq 0$ such that $f\otimes h^n$ extends to $Y$. Since $\Lscr$ and $h$ are Hecke-equivariant, this shows that $M(Y_{K,h})\rightsquigarrow_{\Hcal} M(Y_K)$.
\end{proof}
\begin{corollary}\label{corollary: eo length strata and central leaves}
Let $Y_K$ be an Ekdahl-Oort stratum or a length stratum, and let $C(x)$ be a central leaf that is closed in the Ekedahl-Oort stratum containing it. We have that $M(Y_K)\leftrightsquigarrow_{\Hcal} M(S_K)$ and, if $S_K$ is compact, then $M(S_K) \leftrightsquigarrow_{\Hcal} M((C(x))$.
\end{corollary}
\begin{proof}
Suppose first that $Y_K$ is an Ekedahl-Oort stratum or a length stratum. By \Cref{theorem: general cases}\ref{item: pullback from gzip} we have that $M(\overline{Y}_K)\leftrightsquigarrow_{\Hcal} M(S_K)$ (both for $S_K$ compact and noncompact). By the existence of Hasse invariants, the result follows from \Cref{theorem: general cases}\ref{item: nonvanishing locus}. 
Next let $Z_K$ denote the closure of $C(x)$ in $S_K$ and let $S_w$ be the Ekedahl-Oort stratum containing $C(x)$. By \Cref{theorem: main theorem.in.text} $M(S_K)\leftrightsquigarrow_{\Hcal} M(Z_K)$. Since $C(x)$ is closed in $S_w$, the Hasse invariant $\ha_w$ restricts to a section $h$ on $Z_K$ whose nonvanishing locus is exactly $C(x)$. Thus we apply \Cref{theorem: general cases}\ref{item: nonvanishing locus} again.
\end{proof}
\section*{}
\subsection*{Acknowledgements}
We thank Abhinandan, Wushi Goldring, Wansu Kim and Alex Youcis for helpful discussions. In particular, we thank Wushi Goldring for bringin Serre's Letter to Tate to our attention many years ago, for helpful comments on earlier drafts of this paper, and for encouraging this work. Part of this work was carried out while the author was a JSPS International Research Fellow.
\newpage
\AtNextBibliography{\small}
\printbibliography
\end{document}